\documentclass{article}

\usepackage{amsthm}
\usepackage{amsmath}
\usepackage{amssymb}
\usepackage{tikz}
\usepackage{enumerate}

\usetikzlibrary{trees}

\newtheorem{thm}{Theorem}
\newtheorem{lemma}[thm]{Lemma}
\newtheorem{cor}[thm]{Corollary}
\newtheorem{definition}[thm]{Definition}
\newtheorem{question}{Question}

\newcommand{\N}{\mathbb N}

\begin{document}

\title{Uniform Bounds for Non-negativity of the Diffusion Game}

\author{Andrew Carlotti, Rebekah Herrman}

\maketitle

\begin{abstract}
We study a variant of the chip-firing game called the \emph{diffusion game}. In the diffusion game, we begin with some integer labelling of the vertices of a graph, interpreted as a number of chips on each vertex, and then for each subsequent step every vertex simultaneously fires a chip to each neighbour with fewer chips. In general, this could result in negative vertex labels. Long and Narayanan \cite{period} asked whether there exists an $f(n)$ for each $n$, such that whenever we have a graph on $n$ vertices and an initial allocation with at least $f(n)$ chips on each vertex, then the number of chips on each vertex will remain non-negative. We answer their question in the affirmative, showing further that $f(n)=n-2$ is the best possible bound. We also consider the existence of a similar bound $g(d)$ for each $d$, where $d$ is the maximum degree of the graph.
\end{abstract}

\begin{section}{Introduction}

 In 1986, J. Spencer \cite{balancingvectors} proposed the following solitaire game. Let $N$ chips be arranged in a pile. At each time step, $\lfloor\frac{N}{2}\rfloor$ chips are moved one unit to the right of the pile, and $\lfloor\frac{N}{2}\rfloor$ chips are moved one unit to the left, with one chip remaining in the original pile if $N$ is odd. In subsequent steps, we repeat this process simultaneously on each of the resulting piles. 

 This solitaire game inspired the chip-firing game, introduced by Bj\"orner, Lov\'asz and Shor \cite{chipfiring}. The chip-firing game is played on a simple, connected graph $G$  on the vertex set $[n]=\{1,2,\ldots,n\}$ . In the game, each vertex $v\in [n]$ is assigned an amount of chips, $w_v$. The vertex $v$ is allowed to fire if $w_v \geq d_v$, where $d_v$ denotes the degree of vertex $v$. When vertex $v$ is fired, we remove $d_v$ chips from it, and add one chip to each neighbouring vertex. Only one vertex may be fired at a time, but Bj\"orner et al. found that the order of firings does not affect the length of the game. The game ends when all vertices have fewer chips than neighbours. The chip-firing game has several applications in computer science, mathematics, and physics \cite{algebra,universality,energy,motor}.
 
 The diffusion game was first introduced by Duffy, Lidbetter, Messinger and Nowakowski \cite{introduce} and is a variant of the chip-firing game. In the diffusion game, let $G$ be a graph on the vertex set $[n]$. At time $t=0$ each vertex $v\in [n]$ is assigned an initial integer label $w_v(0)$. We then update all labels at discrete integer time steps according to the rule
\begin{equation*}
w_v(t+1)=w_v(t) + |{u\in \Gamma (v) : w_u(t)>w_v(t)}| - |{u\in \Gamma (v) : w_u(t)<w_v(t)}| \textrm{.}
\end{equation*}
Intuitively, this corresponds to moving one chip along each edge whose vertices have differing numbers of chips, with the vertex with more chips giving a chip to the vertex with fewer.

For each $t\geq 0$, let $w_G(t)$ denote the vector $(w_1(t),w_2(t),\ldots,w_n(t))$. We say that $w_G(t)\ge k$ if $w_v(t)\ge k\ \forall v \in [n]$, and similarly for $w_G(t)\le k$.

Long and Narayanan \cite{period} proved that the diffusion game is eventually periodic with period one or two. That is, there exists $T\in \mathbb{N}$ and $k\in\{1,2\}$ such that for all $t\geq T$, $w_G(t)=w_G(t+k)$.

Our main result answers one of the questions posed by Long and Narayanan in their paper:

\begin{thm} \label{main_theorem}
Let $n \geq 2$. If $w_G(0) \ge f(n) = n-2$, then at all times $t\ge 0$ we have $w_G(t) \ge 0$.
\end{thm}

Indeed, this is the best possible such result, as for each $n\geq 2$, the star on $n$ vertices with $n-3$ chips on each leaf and $n-2$ chips on the central vertex will, after one time step, have $-1$ chips on the central vertex.

We also consider similar bounds based upon the maximum degree $d$ of the graph. We show the following:

\begin{thm}\label{degree_bounds}
Let $g(d)$ be the least possible bound on the minimum number of chips on a vertex such that non-negativity of the labels is guaranteed.
\begin{enumerate}[i]
\item If $d\le 1$, then $g(d) = 0$ \label{degree_1}
\item If $d=2$, then $g(d) = 1$ \label{degree_2}
\item If $d=3$, then $g(d) \ge 3$ \label{degree_3}
\item If $d\ge 4$, then $g(d) = \infty$ \label{degree_4}
\end{enumerate}
\end{thm}

For $d=3$, we know only that $g(3)\ge 3$; it may be that this inequality is tight.

\end{section}

\begin{section}{Order-based Bounds}

We will proceed by defining the \emph{weak diffusion game}, a more general, non-deterministic variant of the original diffusion game. We then reduce the problem to considering a specific initial state, and show that subsequent states can be represented by a digraph encoding, which need not be unique.

The following weaker result can be obtained by a conceptually simpler version of our main proof. This version differs from the one presented in two ways: the digraph encoding used does not require edge weights, and we need only reduce to the initial state $(n-1,n-1,\ldots,n-1)$. One may wish to consider this variation as a stepping stone to understanding the full proof of Theorem \ref{main_theorem}.

\begin{thm} \label{weaker_theorem}
Let $G$ be a graph with $n$ vertices. If $w_G(0) \ge n-1$, then at all times $t\ge 0$ we have $w_G(t) \ge 0$.
\end{thm}

\begin{subsection}{The Weak Diffusion Game}

We begin by making two modifications to the diffusion process.

First, rather than transferring chips along edges of a predetermined constant graph, we instead may choose at each time step whether or not to allow a chip to transfer between each pair of vertices. That is, at each time step, for each pair of vertices $u$ and $v$ with $w_u(t)>w_v(t)$, we are allowed to choose whether or not a chip is transferred from vertex $u$ to vertex $v$ (with these transfers being the only transfers allowed). So the original diffusion process is now one of many possible evolutions of the labels $w_G$.

Second, we permit also the transfer of chips between vertices having equal numbers of chips.

These modifications give us a process we shall call the \emph{weak diffusion game}. We can represent our choices of when to move chips by the values $d_{uv}(t)$ ($u,v\in [n]$, $u\ne v$, $t\in \N$), which satisfy:
\begin{gather*}
d_{uv}(t) \in \{-1,0,1\} \\
d_{uv}(t) = -d_{vu}(t) \\
d_{uv}(t)(w_u(t-1)-w_v(t-1)) \ge 0
\end{gather*}
The labels then evolve according to:
\begin{equation*}
w_v(t)=w_v(t-1)+\sum_{u\ne v} d_{uv}(t)
\end{equation*}

We can now state the following theorem

\begin{thm} \label{monotonicity}
Let $G$ be a graph with $n$ vertices. Suppose that $w_G$ is a possible evolution of the weak diffusion game. Then, given any initial state $w'_G(0)\leq w_G(0)$, there exists an evolution of the weak diffusion game $w'_G$ with this initial state, and a sequence of permutations $P_t$, such that for each $t\geq 0$ and $u\in [n]$, we have $w'_{P_t(u)}(t)\leq w_u(t)$. That is, if we remove some chips from the initial state of some evolution then, up to a permutation of the vertex labels at each time step, we can then remove chips from later states to obtain another valid evolution without ever needing to add chips to a vertex.
\end{thm}

\begin{proof}
It will suffice to prove this for a removal of one chip from the initial state; the full result then follows by induction on the number of chips removed.

Furthermore, it will suffice to show this for one time step; the result will then follow by induction on $t$.

Without loss of generality, we may assume that the transfer graph represented by $d_{uv}(1)$ is acyclic (since transfers forming a cycle have no net effect on the distribution of chips). We may then assume that the vertices are labelled such that if $1 \leq u < v \leq n$, then $w_u(0)\geq w_v(0)$ and $d_{uv}(1) \neq -1$. Now suppose $w'_G(0)$ is obtained from $w_G(0)$ by removing one chip from vertex $k$. Let $k'=max\{i\in[n]|w_i(0)=w_k(0)\}$, set the permutation $P_1=P=(kk')$, and set $d'_{P(u)P(v)}(1)=d_{uv}(1)$. Applying the transfers represented by $d'(1)$ gives us $w'_G(1)$ satisfying $w'_{P_1(u)}(1)\leq w_u(1)$, as required.
\end{proof}

\end{subsection}

\begin{subsection}{Proof of Theorem \ref{main_theorem}}

We begin by providing a link between bounds in the weak diffusion game and bounds in the original diffusion game, reducing the problem to establishing non-negativity of the weak diffusion game with specific initial conditions. We then produce an encoding of the game in a sequence of weighted directed graphs, leading to non-negativity as an immediate consequence.

\begin{lemma}
Let $G$ be a graph with $n$ vertices, let $k\geq 0$ and let $w_G$ be an evolution of the original diffusion process with $w_G(0)\geq k$. Suppose that, for some $t\geq 0$, we have $w_G(t) \ngeq 0$. Then there exists an evolution $w'_G$ of the weak diffusion game with $w'_G(0)=(k+1,k,k,\ldots,k)$ and $w'_G(t)\ngeq 0$.
\end{lemma}

\begin{proof}
Note that $w_G$ is automatically a valid evolution of the weak diffusion game. Furthermore, we cannot have $w_v(0)=k\ \forall v$, as otherwise $w_G$ would be constant, contradicting $w_u(t)<0$. The lemma then follows from Theorem \ref{monotonicity} with an initial relabelling of the vertices.
\end{proof}

It now suffices to show that the weak diffusion game with initial state $w_G(0)=(n-1,n-2,\ldots,n-2)$ must remain non-negative.

\begin{definition}
Let $G$ be an $n$-vertex graph, and let $w_G$ be an evolution of the weak diffusion game on $G$ with mean label $\mu =\sum_v{w_v(0)} /n$. A \emph{digraph encoding} of a state $w_G(t)$ is a weighted directed graph with edge weights $\lambda_{uv}(t)$ for each $u,v \in [n]$, satisfying:
\begin{gather*}
\lambda_{uv}(t) \in [-1,1]\cap \mathbb{Z}/n\\
\lambda_{uv}(t)=-\lambda_{vu}(t)\\
w_v(t)=\mu+\sum_u\lambda_{uv}(t)
\end{gather*}
\end{definition}

Some digraph encodings lead more naturally to a representation of the subsequent state. This property is captured in the following definition:

\begin{definition}
Let $w_G$ be as above. We say that an encoding of the state $w_G(t)$ is \emph{good} if $\lambda_{uv}(t)\le 0$ whenever $w_u(t)\ge w_v(t)$. Otherwise, we say that the encoding is \emph{bad}.
\end{definition}

Note that the existence of a digraph encoding for a state $w_G(t)$ bounds the number of chips on each vertex between $\mu - (n-1)$ and $\mu + (n-1)$. We aim to show that a digraph encoding exists for every state of our evolution $w_G$. The following lemma will facilitate this:

\begin{lemma}
Let $w_G$ be as above. If $w_G(t)$ has a digraph encoding, then it has a good digraph encoding.
\end{lemma}

\begin{proof}
Of the many possible digraph encodings for $w_G(t)$, consider an encoding of least absolute sum---that is, an encoding $\lambda_{uv}(t)$ in which $\sum_{u<v}|\lambda_{uv}(t)|$ is minimised. We show that this is necessarily a good encoding.

For a contradiction, suppose instead that this encoding is bad. Then there exist $u,v$ such that $w_u(t)\ge w_v(t)$, but $\lambda_{uv}(t)>0$. We then have:
\begin{gather*}
\sum_{w\ne u,v}\left(\lambda_{wu}(t)-\lambda_{wv}(t)\right) + \lambda_{vu}(t)-\lambda_{uv}(t)= w_u(t)- w_v(t)\\
\sum_{w\ne u,v}\left(\lambda_{wu}(t)-\lambda_{wv}(t)\right) > 0
\end{gather*}
So there exists $w$ such that:
\begin{gather*}
\lambda_{wu}(t)-\lambda_{wv}(t) > 0\
\end{gather*}
Now let $a=\lambda_{uv}(t)$, $b=\lambda_{vw}(t)$ and $c=\lambda_{wu}(t)$. We have that $a>0$ and $b+c>0$. Note that we can add a constant $k$ to each of these terms without affecting the encoded vertex labels. Since this was an encoding that minimised the absolute sum, we have that $|a+k|+|b+k|+|c+k|$ is minimised at $k=0$ (subject to $a+k,b+k,c+k\in [-1,1]$). Since at least two out of $a$, $b$ and $c$ are positive, and none of them are equal to $-1$, it is clear that taking $k=-1/n$ reduces the sum of the absolute values without breaking any of the constraints. Thus there is a digraph encoding with smaller absolute sum, contradicting the minimality of the original encoding.

It follows that the original encoding was good, as desired.
\end{proof}

We can now show the existence of encodings at all time steps:

\begin{lemma}
Let $w_G$ be as above. Whenever $w_G(t)$ has a digraph encoding, then $w_G(t+1)$ has a digraph encoding.
\end{lemma}

\begin{proof}
By the previous lemma, we may take a good encoding $\lambda_{uv}(t)$ of $w_G(t)$. Then $\lambda_{uv}(t+1)=\lambda_{uv}(t)+d_{uv}(t+1)$ gives an encoding of $w_G(t+1)$. In particular, $\lambda_{uv}(t+1)\in \{-1,0,1\}$, since $d_{uv}(t+1)>0$ implies $w_u(t)\ge w_v(t)$, which in turn implies $\lambda_{uv}(t)\le 0$.
\end{proof}

\begin{cor}\label{encodability}
Whenver $w_G(0)$ has a digraph encoding, then $w_G(t)$ has a digraph encoding for all $t\ge 0$.\qed
\end{cor}

We can now complete our proof of Theorem \ref{main_theorem}. First, observe that $w_G(0)=(n-1,n-2,\ldots,n-2)$ has a digraph encoding where $\mu = n-2+(1/n)$, $\lambda_{u,1}(0)=1/n$ and $\lambda_{uv}(0)=0$ for $u,v\ne 1$. Thus by Corollary \ref{encodability}, $w_G(t)$ has a digraph encoding for all $t\ge 0$.

This means that for any $v,t$, we have
\begin{equation*}
w_v(t)=\mu - \sum_{u\ne v}\lambda_{uv}(t) \ge \mu - (n-1) = -1+(1/n)
\end{equation*}

Since $w_v(t)$ is an integer, this implies $w_v(t)\ge 0$, as required.\qed

\end{subsection}

\begin{subsection}{Remarks}

The proof of Theorem \ref{main_theorem} applies also to directed graphs and to graphs which vary over time. We can further extend it to multigraphs; in this case, if $m$ is the maximum number of edges between two vertices, and our initial state has at least $m(d-1)-1$ chips on each vertex, then no vertex ever attains a negative number of chips.

The idea of digraph encodings can also be used to give an alternative proof of Long and Narayanan's result that the diffusion game is eventually periodic (although this method does not bound the eventual period as strongly). Indeed, we extend the definition of a digraph encoding to allow edges weights to take any value in $\mathbb R$. Then we encode a state using the digraph whose edge-weight sequence, ordered from largest to smallest, is lexicographically smallest. These edge-weight sequences form a sequence over time, which is decreasing (in the above order) until all the weights have magnitude less than $1$. This happens in finite time since every edge weight is in $\mathbb Z / k$ for some $k=k(n)$, after which an ordinary digraph encoding exists for every state.

\end{subsection}

\end{section}

\begin{section}{Bounds using the Maximum Degree}

We now prove the bounds given in Theorem \ref{degree_bounds}. Note that we may restrict our attention to infinite $d$-regular trees. Indeed, for any graph $G$ with maximum degree $d$, take a disjoint union of two copies of $G$, and add edges between corresponding vertices in the two copies to make the graph $d$-regular. Then consider the universal cover $H$ of $G$ -- this is the $d$-regular infinite tree. We may assign labels to $H$ according to the labels of the corresponding vertices of $G$; these labels evolve in the same manner as the corresponding labels on $G$. Conversely, if any vertex $v$ of a $d$-regular tree can attain a negative label in finite time $T$, then this will be achieved also with the same initial conditions restricted to the finite graph consisting of all vertices at distance at most $T$ from $v$.

\begin{subsection}{Proof of Theorem \ref{degree_bounds}}

We consider each case in turn:

\begin{enumerate}[i]

\item $d\leq 1$

The graph $G$ is either a point or a single edge; in either case the result is trivial.

\item $d=2$

First, note that  $g(2)>0$, as a path on three vertices starting with a single chip on the central vertex attains a negative chip value on the second diffusion step.

Next, consider diffusion on the infinite path with vertex set $V=\mathbb Z$, and assume that all labels are initially at least $1$. Suppose for contradiction that some label subsequently becomes negative, and let $T_0$ be the earliest time at which any vertex has a negative label. We will us the following lemma: 

\begin{lemma}\label{consecutive_zeroes}
Before time $T_0$, no vertex can have label $0$ on two consecutive time steps. 
\end{lemma}

\begin{proof}
Suppose for a contradiction that $w_v(T-1)=w_v(T)=0$ for some $v\in V$ and $0<T<T_0$. Take the least such $T$.

Then $T>1$ since $w_v(0) > 0\ \forall \ v$, and $w_{v-1}(T-1),w_{v+1}(T-1) \ge 0$, as $T-1<T_0$.

It follows that $w_{v-1}(T-1)=w_{v+1}(T-1)=0$, else the diffusion process would yield $w_v(T)>0$.

Finally, $w_v(T-2)=0$, otherwise $w_{v-1}(T-1)$, $w_v(T-1)$ and $w_{v+1}(T-1)$ could not all be $0$.

This contradicts the minimality of $T$, yielding the desired result.
\end{proof}

Now let $T_1$ be the least time such that there exists a vertex $v_1$ with:
\begin{equation*}
(w_{v_1-1}(T_1), w_{v_1}(T_1), w_{v_1+1}(T_1))= (0,1,1) \text{ or } (1,1,0)
\end{equation*}

If the patterns $(0,1,1)$ and $(1,1,0)$ do not exist, then we say that $T_1=\infty$. We have the following lemma about the labels that precede a zero:

\begin{lemma}\label{before_zero}
Let $T\leq T_0,T_1$, and $w_v(T)=0$. Then $w_v(T-1)=2$.
\end{lemma}

\begin{proof}
By the definition of $T_0$, we have $w_v(T-1)\geq 0$. Lemma \ref{consecutive_zeroes} tells us that $w_v(T-1) \neq 0$. We also have $w_v(T-1)\neq 1$, else, by the definition of $T_0$, we would need $w_{v-1}(T-1)$ and $w_{v+1}(T-1)$ to equal $0$ and $1$ in some order, contradicting the definition of $T_1$. Since $w_v$ can change by at most $2$ at each step of the diffusion process, it follows that $w_v(T-1)=2$.
\end{proof}

We now show that the pattern $(0,1,1)$ or $(1,1,0)$ exists before time $T_0$.

\begin{lemma}
$T_1<T_0$.
\end{lemma}

\begin{proof}
Suppose for a contradiction that the pattern $(0,1,1)$ or $(1,1,0)$ does not exist before time $T_0$. We work backwards from time $T_0$. By the definition of $T_0$ and $v_0$, it follows that $w_{v_0-1}(T_0-1)=w_{v_0+1}(T_0-1)=0$ and $w_{v_0}(T_0-1)=1$.

Now consider time $T_0-2$. By Lemma \ref{before_zero}, we have that $w_{v_0-1}(T_0-2)=w_{v_0+1}(T_0-2)=2$. But then the diffusion process cannot attain $w_{v_0}(T_1)=1$.

Hence the pattern $(0,1,1)$ or $(1,1,0)$ does exist before time $T_0$.
\end{proof}

We shall finish by working backwards from time $T_1$ until we reach another contradiction. We assume that $(w_{v_1-1}(T_1), w_{v_1}(T_1), w_{v_1+1}(T_1))=(0,1,1)$. Then by Lemma \ref{before_zero} we have that $w_{v_1-1}(T_1-1)=2$. In order that the diffusion process gives us the stated values at time $T_1$, we require that $(w_{v_1}(T_1-1), w_{v_1+1}(T_1-1), w_{v_1+2}(T_1-1)) = (1,0,0) \text{ or } (0,0,w)$, for some $w\in \mathbb{N}$. In either case we have two adjacent $0$'s, which by Lemma \ref{before_zero} must each be preceded by a $2$. However, adjacent $2$'s cannot become adjacent $0$'s under one step of the diffusion process.

Having derived a contradiction from our original assumption, we conclude that no vertex label can ever become negative. \qed

\item $d=3$

The following diagrams demonstrate that $f(3)\ge 3$. Note that the initial state uses only two different labels: $2$ and $3$.
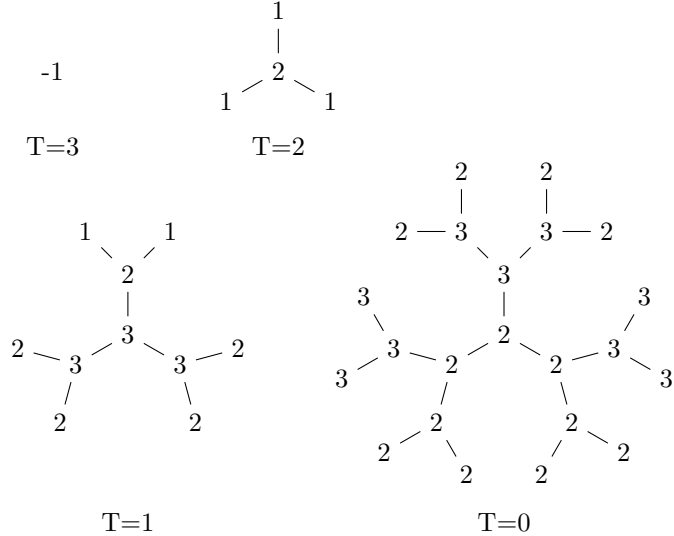
\begin{center}
\tikz [grow cyclic, level distance=8mm,
    level 1/.style={sibling angle=120,rotate=-150},
    level 2/.style={sibling angle=90},
    ] {
    \node at (1,8.5) {-1};
    \node at (1,7.5) {T=3};

    \node at (4,8.5) {2}
        child { node {1} }
        child { node {1} }
        child { node {1} };
    \node at (4,7.5) {T=2};

    \node at (2,5) {3}
        child { node {2} 
            child { node {1} }
            child { node {1} }
        }
        child { node {3} 
            child { node {2} }
            child { node {2} }
        }
        child { node {3} 
            child { node {2} }
            child { node {2} }
        };
    \node at (2,2.5) {T=1};

    \node at (7,5) {2}
        child { node {3} 
            child { node {3}
                child { node {2} }
                child { node {2} }
            }
            child { node {3}
                child { node {2} }
                child { node {2} }
            }
        }
        child { node {2} 
            child { node {3}
                child { node {3} }
                child { node {3} }
            }
            child { node {2}
                child { node {2} }
                child { node {2} }
            }
        }
        child { node {2} 
            child { node {2}
                child { node {2} }
                child { node {2} }
            }
            child { node {3}
                child { node {3} }
                child { node {3} }
            }
        };
    \node at (7,2.5) {T=0};
}
\end{center}

\item $d=4$

Consider the infinite $d$-regular tree, and fix some vertex $v_0$. We assign labels to each vertex according to its distance from $v_0$; in particular, at time $t$, we assign the label $w_i(t)$ to all vertices at distance $i$ from $v_0$.

Working backwards from a time $T$ at which $w_i(T)=-1$, we can construct the following evolution:

\begin{equation*}
\begin{array}{c|ccccc}
t & T & T-1 & T-2 & \ldots & 0 \\
\hline
w_0(t) & -1 & d-1 & 2d-1 & \ldots & Td-1\\
w_1(t) & & d-2 & 2d-4 & \ldots & Td-2T \\
w_2(t) & & & 2d-5 & \ldots & Td-2T-1 \\
\vdots & & & & \ddots & \vdots \\
w_T(t) & & & & & Td-3T+1 \\
\end{array}
\end{equation*}

Thus $g(d)>Td-3T+1$ for all $T>0$, so $g(d)=\infty$, as required.

\end{enumerate}

\end{subsection}

\end{section}

\begin{section}{Concluding Remarks}

Our results on maximum degree bounds are incomplete; specifically, we leave the following unanswered:

\begin{question}
What is $g(3)$? In particular, is it finite?
\end{question}

More generally, when $g(d)$ was found to be infinite, we needed to use arbitraily large ranges of initial labels in order to attain negative labels for a given minimum initial label. 

This raises the following question, originally asked by Long and Narayanan in the equivalent context of infinite graphs of bounded degree:

\begin{question}
Does there exist $g(d,k)<\infty$ such that for any graph $G$ of maximum degree $d$, if the vertices of $G$ are given initial labels in $[g(d,k), g(d,k)+k]$, then all vertex labels in this diffusion game remain non-negative?
\end{question}

\end{section}

\end{document}